\documentclass{amsart}
\usepackage{latexsym}
\usepackage{amssymb, amsmath}
\usepackage{graphicx}
\usepackage{epstopdf}
\usepackage{color}

\newtheorem{theorem}{Theorem}[section]      
\newtheorem{definition}[theorem]{Definition} 
\newtheorem{lemma}[theorem]{Lemma}       
\newtheorem{proposition}[theorem]{Proposition}
\newtheorem{corollary}[theorem]{Corollary}

\newtheorem*{convention*}{Convention}

\newcommand{\A}{\mathcal A}

\newcommand{\B}{\mathcal B}

\newcommand{\D}{\mathcal D}
\newcommand{\G}{\mathcal G}
\newcommand{\F}{\mathcal F}

\newcommand{\N}{\mathbb N}

\renewcommand{\P}{\mathcal P}

\renewcommand{\S}{\mathcal S}

\begin{document}

\title{Unary Automatic Graphs: \ \\  An Algorithmic Perspective}
\author[B. Khoussainov, J. Liu and M. Minnes]{Bakhadyr
Khoussainov$^1$, Jiamou Liu$^1$ and Mia Minnes$^2$\\
$^1$ Department of Computer Science, University of Auckland,
Auckland, New Zealand.\\ $^2$ Department of
Mathematics, Cornell University, Ithaca, NY, USA.}


\begin{abstract}
 This paper studies infinite graphs
produced from a natural unfolding operation applied to finite
graphs. Graphs produced via such operations are of finite degree
and automatic over the unary alphabet (that is, they can be
described by finite automata over unary alphabet). We investigate
algorithmic properties of such unfolded graphs given their finite
presentations. In particular, we ask whether a given node belongs
to an infinite component, whether two given nodes in the graph are
reachable from one another, and whether the graph is connected. We
give polynomial-time algorithms for each of these questions. For a
fixed input graph, the algorithm for the first question is in
constant time and the second question is decided using an
automaton that recognizes reachability relation in a uniform way.
Hence, we improve on previous work, in which non-elementary or
non-uniform algorithms were found.
\end{abstract}

\maketitle

\section{Introduction}

We study the algorithmic properties of infinite graphs that result
from a natural unfolding operation applied to finite graphs. The
unfolding process always produces infinite graphs of finite
degree. Moreover, the class of resulting graphs is a subclass of
the class of automatic graphs. As such, any element of this class
possesses all the known algorithmic and algebraic properties of
automatic structures. An equivalent way to describe these graphs
employs automata over a unary alphabet (see Theorem
\ref{thm:Gsigma}). Therefore, we call this class of graphs
\emph{unary automatic graphs of finite degree}.

In recent years there has been increasing interest in the study of
structures that can be presented by automata. The underlying idea
in this line of research  consists of using automata (such as word
automata,  B\"uchi automata, tree automata, and Rabin automata) to
represent structures and study logical and algorithmic
consequences of such presentations. Informally, a structure
$\A=(A; R_0, \ldots, R_m)$ is {\em automatic} if the domain $A$
and all the relations $R_0$, $\ldots$, $R_m$ of the structure are
recognized by finite automata (precise definitions are in the next
section). For instance, an automatic graph is one whose set of
vertices and set of edges can each be recognized by finite
automata. The idea of automatic structures was initially
introduced by Hodgson \cite{hodgson} and was later rediscovered by
Khoussainov and Nerode \cite{kn}. Automatic structures possess a
number of nice algorithmic and model-theoretic properties.  For
example, Khoussainov and Nerode proved that the first-order theory
of any automatic structure is decidable \cite{kn}. This result is
extended  by adding the $\exists^{\infty}$ (there are infinitely
many) and $\exists^{n,m}$ (there are $m$ many mod $n$) quantifiers
to the first order logic \cite{gradel, krs}. Blumensath and
Gr\"{a}del proved a  logical characterization theorem stating that
automatic structures are exactly those definable in the following
fragment of the arithmetic $(\omega; +, |_2, \leq, 0)$, where $+$
and $\leq$ have their usual meanings and $|_2$ is a weak
divisibility predicate for which $x|_2 y$ if and only if $x$ is a
power of $2$ and divides $y$ \cite{gradel}. Automatic structures
are closed under first-order interpretations. There are
descriptions of automatic linear orders and trees in terms of
model theoretic concepts such as Cantor-Bendixson ranks
\cite{rubin}. Also, Khoussainov, Nies, Rubin and Stephan have
characterized
 the isomorphism types of automatic Boolean
algebras \cite{kns}; Thomas and Oliver have given a full
description of finitely generated automatic groups \cite{oliver
thomas}. Some of these results have direct algorithmic
implications. For example, isomorphism problem for automatic
well-ordered sets and Boolean algebras is decidable \cite{kns}.

\smallskip

There is also a body of work devoted to the study  of
resource-bounded complexity of the first order theories of
automatic structures. For example, on the one hand, Gr\"{a}del and
Blumensath constructed examples of automatic structures whose
first-order theories are non-elementary \cite{gradel}. On the
other hand, Lohrey in \cite{lohrey} proved that the
first-order theory of any automatic graph of bounded degree is
elementary.  It is worth noting that when both a first-order
formula and an automatic structure $\A$ are fixed, determining if
a tuple $\bar{a}$ from $\A$ satisfies $\phi(\bar{x})$ can be done
in linear time.

\smallskip

Most of the results about automatic structures,  including the
ones mentioned above, demonstrate that in various concrete senses
automatic structures are not complex from a logical point of view.
However, this intuition can be misleading. For example, in
\cite{kns} it is shown that the isomorphism problem for automatic
structures is $\Sigma_1^1$-complete. This informally tells us that
there is no hope for a description (in a natural logical language)
of the isomorphism types of automatic structures. Also,
Khoussainov and Minnes \cite{km} provide examples of automatic
structures whose Scott ranks can be as high as possible, fully
covering the interval $[1,\omega_1^{CK}+1]$ of ordinals (where
$\omega_1^{CK}$ is the first non-computable ordinal).  They also
show that the ordinal heights of well-founded automatic relations
can be arbitrarily large ordinals below $\omega_1^{CK}$.

\smallskip

In this paper, we study the class of unary automatic graphs of
finite degree. Since these graphs are described by the unfolding
operation (Definition \ref{dfn: unfolding}) on the pair of finite
graphs $(\D, \F)$, we use this pair to represent the graph. The
size of this pair is the sum of the sizes of the automata  that
represent these graphs. In the study of algorithmic properties of
these graphs one directly deals with the pair $(\D, \F)$. We are
interested in the following natural decision problems:

\begin{itemize}
\item {\bf Connectivity Problem}. Given an automatic  graph $\G$,
decide if $\G$ is connected. \item {\bf Reachability Problem}.
Given an automatic graph $\G$  and two vertices $x$ and $y$ of the
graph, decide if there is a path from $x$ to $y$.
\end{itemize}

If we restrict to the class of finite graphs, these two problems
are decidable and can be solved in linear time on the sizes of the
graphs.  However, we are interested in infinite graphs and
therefore much more work is needed to investigate the problems
above.  In addition, we also pose the following two problems:

\begin{itemize}

\item  {\bf Infinity Testing Problem}. Given an  automatic graph
$\G$ and a vertex $x$, decide if the component of $\G$ containing
$x$ is infinite.

\item {\bf Infinite Component Problem}. Given an automatic graph
$\G$ decide if $\G$ has an infinite component.

\end{itemize}

Unfortunately, for the class of automatic graphs all of the above
problems are undecidable. In fact, one can provide exact bounds on
this undecidability. The connectivity problem is
$\Pi_2^0$-complete; the reachability problem is
$\Sigma_1^0$-complete; the infinite component problem is
$\Sigma_3^0$-complete; and the infinity testing problem is
$\Pi_2^0$-complete \cite{rubin}.

Since all unary automatic structures are first-order definable in
$S1S$ (the monadic second-order logic of the successor function),
it is not hard to prove that all the problems above are decidable
\cite{Blumensath, rubin}. Direct constructions using this
definability in $S1S$ yield algorithms with non-elementary time
since one needs to transform $S1S$ formulas into automata
\cite{buchi}.  However,  we provide polynomial-time algorithms for
solving all the above problems for this class of graphs. We now
outline the rest of this paper by explaining the main results.
We comment that these polynomial-time algorithms are based on
deterministic input automata.

\smallskip

Section 2 introduces the main definitions needed, including the
concept of automatic structure. Section 3 singles out unary
automatic graphs and provides a characterization theorem (Theorem
\ref{thm:characterization}).
Section 4 introduces unary automatic graphs of finite degree. The
main result is Theorem \ref{thm:Gsigma} that explicitly provides
an algorithm for building unary automatic graphs of finite degree.
This theorem is used throughout the paper. Section 5 is devoted to
deciding the infinite component problem. The main result is the
following:

\smallskip

\noindent{\bf Theorem \ref{thm:infinite component}} {\em The
infinite component problem for unary automatic graph of finite
degree $\G$ is solved in $O(n^3)$, where $n$ is the number of
states of the deterministic finite automaton recognizing $\G$.}

\smallskip

In this section, we make use of the concept of oriented walk for
finite directed graphs. The subsequent section is devoted to
deciding the infinity testing problem. The main result is the
following:

\smallskip

\noindent{\bf Theorem \ref{thm:InfTest}} \ {\em The infinity
testing problem for unary automatic graph of finite degree $\G$ is
solved in $O(n^3)$, where $n$ is the number of states of the deterministic finite automaton $\A$ recognizing $\G$. In particular, when $\A$ is
fixed, there is a constant time algorithm that decides the
infinity testing problem on $\G$.}

\smallskip

The fact that there is a constant time algorithm when $\A$ is
fixed will be made clear in the proof. The value of the constant
is polynomial in the number of states of $\A$.

\smallskip

The reachability problem is addressed in Section 7. This problem
has been studied in \cite{bouajjani},\cite{Esparza}, \cite{Thomas02} via the
class of  {\bf pushdown graphs}. A pushdown graph is the
configuration space of a pushdown automaton. Unary automatic
graphs are pushdown graphs \cite{Thomas02}. In \cite{bouajjani,
Esparza, Thomas02} it is proved that for a pushdown graph $\G$,
given a node $v$, there is an automaton that recognizes all nodes
reachable from $v$. The number of states in the automaton depends
on the input node $v$. This result implies that there is an
algorithm that decides the reachability problem on unary automatic
graphs of finite degree. However, there are several issues with
this algorithm. The automata constructed by the algorithm are not
uniform in $v$ in the sense that different automata are built for
different input nodes $v$. Moreover, the automata are
nondeterministic. Hence, the size of the deterministic equivalent
automata is exponential in the size of the representation of $v$.
Section 7 provides an alternative algorithm to solve
 the reachability problem on unary automatic graphs of finite degree uniformly.
This new algorithm constructs a deterministic automaton
$\A_{Reach}$ that accepts the set of pairs $\{(u,v) \mid $ there
is a path from $u$ to $v\}$.  The size of $\A_{Reach}$ only
depends on the number of states of the automaton $n$, and
constructing the automaton requires polynomial-time in $n$. The
practical advantage of such a uniform solution is that, when
$\A_{Reach}$ is built, deciding whether node $v$ is reachable from
$u$ by a path takes only linear time (details are in Section 7).
The main result of this section is the following:

\smallskip

\noindent{\bf Theorem \ref{thm:reachability}} \ {\em Suppose $\G$
is a unary automatic graph of finite degree represented by deterministic finite
automaton $\A$ of size $n$. There exists a polynomial-time
algorithm that solves the reachability problem on $\G$. For inputs
$u,v$, the running time of the algorithm is $O(|u| + |v| + n^4)$.}

\smallskip

Finally, the last  section  solves the connectivity
problem for $\G$.

\smallskip

\noindent{\bf Theorem \ref{thm:connectivity}} {\em The
connectivity problem for unary automatic graph of finite degree
$\G$ is solved in $O(n^3)$, where $n$ is the number of states of
the deterministic finite automaton recognizing $\G$.}

The authors would like to thank referees for comments on improvement of this paper.

\section{Preliminaries}

A {\bf finite automaton} $\A$ over an alphabet $\Sigma$ is a tuple
$(S,\iota,\Delta,F)$, where $S$ is a finite set of {\bf states},
$\iota \in S$ is the {\bf initial state}, $\Delta \subset S \times
\Sigma \times S$ is the {\bf transition table} and $F \subset S$
is the set of {\bf final states}. A {\bf computation} of $\A$ on a
word $\sigma_1 \sigma_2 \dots \sigma_n$ ($\sigma_i \in \Sigma$) is
a sequence of states, say $q_0,q_1,\dots,q_n$, such that $q_0 =
\iota$ and $(q_i,\sigma_{i+1},q_{i+1}) \in \Delta$ for all $i \in
\{0,1,\ldots,n-1\}$.  If $q_n \in F$, then the computation is {
\bf successful} and we say that automaton $\A$ {\bf accepts} the
word. The {\bf language} accepted by the automaton $\A$ is the set
of all words accepted by $\A$. In general, $D \subset
\Sigma^{\star}$ is {\bf FA recognizable}, or {\bf regular}, if $D$
is the language accepted by some finite automaton. In this paper
we always assume the automata are deterministic.  For two states
$q_0, q_1$, the {\bf distance} from $q_0$ to $q_1$ is the minimum
number of transitions required for $\A$ to go from $q_0$ to $q_1$.

\smallskip

To formalize the notion of a relation being recognized by an
automaton, we define synchronous $n$-tape automata.  Such an
automaton can be thought of as a one-way Turing machine with $n$
input tapes.  Each tape is semi-infinite having written on it a
word in the alphabet $\Sigma$ followed by a succession of
$\diamond$ symbols. The automaton starts in the initial state,
reads simultaneously the first symbol of each tape, changes state,
reads simultaneously the second symbol of each tape, changes
state, etc., until it reads $\diamond$ on each tape. The automaton
then stops and accepts the $n$-tuple of words if and only if it is
in a final state.

\smallskip

More formally, we write $\Sigma_{\diamond}$ for $\Sigma \cup
\{\diamond\}$ where $\diamond$ is a symbol not in $\Sigma$. The
{\bf convolution} of a tuple $(w_1,\cdots,w_n) \in \Sigma^{\star
n}$ is the string $\otimes(w_1,\cdots,w_n)$ of length $\max_i
|w_i|$ over the alphabet $(\Sigma_{\diamond})^n$ which is defined
as follows: the $k^{th}$ symbol is $(\sigma_1,\ldots,\sigma_n)$
where $\sigma_i$ is the $k^{th}$ symbol of $w_i$ if $k \leq
|w_i|$, and is $\diamond$ otherwise.  The {\bf convolution} of a
relation $R \subset \Sigma^{\star n}$ is the relation $\otimes R
\subset (\Sigma_{\diamond})^{n \star}$ formed as the set of
convolutions of all the tuples in $R$.  
An $n$-ary relation $R \subset
\Sigma^{\star n}$ is {\bf FA recognizable}, or {\bf regular}, if
its convolution $\otimes R$ is recognizable by a finite
automaton.

\smallskip

A {\bf structure} $\S$ consists of a countable set $D$ called the
{\bf domain} and some relations and operations on $D$. We may
assume that $\S$ only contains relational predicates since
operations can be replaced with their graphs.  We write
$\S=(D,R_1^D, \ldots, R_k^D, \ldots )$ where $R_i^D$ is an
$n_i$-ary relation on $D$. The relation $R_i$ are sometimes called
basic or atomic relations. We assume that the function $i\mapsto
n_i$ is always a computable one. A structure $\S$ is {\bf
automatic over alphabet $\Sigma$} if its domain $D \subset
\Sigma^{\star}$ is finite automaton recognizable, and there is an
algorithm that for each $i$ produces an $n_{i}$-tape automaton
recognizing the relation $R_i^D \subset (\Sigma^{\star})^{n_i}$. A
structure is called {\bf automatic} if it is automatic over some
alphabet. If $\B$ is isomorphic to an automatic structure $\S$,
then we call $\S$ an {\bf automatic presentation} of $\B$ and say
that $\B$ is {\em \bf automatically presentable}.

\smallskip

An example of an automatic structure is the word structure
$(\{0,1\}^\star, L,R, E,\preceq)$, where for all $x,y \in
\{0,1\}^\star$, $L(x)=x0$, $R(x)=x1$, $E(x,y)$  if and only if
$|x|=|y|$, and $\preceq$ is the lexicographical order. The
configuration graph of any Turing machine is another example of an
automatic structure. Examples of automatically presentable
structures are $(\N, +)$, $(\N, \leq)$, $(\N, S)$,  the group
$(Z,+)$, the order on the rational $(Q, \leq)$, and the Boolean
algebra of finite and co-finite subsets of $\N$. Consider the
first-order logic extended by $\exists^{\omega}$ (there exist
infinitely many) and $\exists^{n,m}$ (there exist $n$ many mod
$m$, where $n$ and $m$ are natural numbers) quantifiers. We denote
this logic by $FO+\exists ^{\omega} +\exists^{n,m}$. We will use
the following theorem without explicit reference to it.

\begin{theorem} {\rm \cite{kn}}
Let $\A$ be an automatic structure. There exists an algorithm
that, given a formula $\phi(\bar{x})$ in $FO+\exists ^{\omega}
+\exists^{n,m}$, produces an automaton that recognizes exactly
those tuples $\bar{a}$ from the structure that make $\phi$ true.
In particular, the set of all sentences of $FO+\exists ^{\omega}
+\exists^{n,m}$ which are true in $\A$ is decidable.
\end{theorem}

\section{Unary automatic graphs}

We now turn our attention to the subclass of the automatic
structures which is the focus of the paper.

\begin{definition}
A structure $\A$ is {\bf unary automatic} if it has an automatic
presentation whose domain is $1^\star$ and whose relations are
automatic.
\end{definition}

Examples of unary automatic structures are $(\omega, S)$ and
$(\omega, \leq)$. Some recent work on unary automatic structures
includes a characterization of unary automatic linearly ordered
sets, permutation structures, graphs, and equivalence structures
\cite{sasha-bakh, Blumensath}. For example, unary automatic
linearly ordered sets are exactly those that are isomorphic to a
finite sum of orders of type $\omega$, $\omega^\star$ (the order
of negative integers), and finite $n$.

\smallskip

\begin{definition} A {\bf unary automatic graph} is a graph $(V, E)$ whose
domain is $1^\star$, and whose edge relation $E$ is regular.
 \end{definition}

We use the following example to illustrate that this class of graphs are the best possible. Consider the class of graphs with all vertices being of the form $1^*2^*$ for some alphabet $\Sigma = \{1,2\}$. At first sight, graphs of this form may have an intermediate position between unary and general automatic graphs. However, the infinite grid $G_2 = \{\N\times \N, \{((i,j),(i,j+1)) \mid i,j\in \N\}, \{((i,j), (i+1,j)) \mid i,j\in \N\}\}$ can be coded automatically over $1^*2^*$ by $(i,j) \rightarrow 1^i2^j$, and $MSO(G_2)$ is not decidable \cite{wohrle04}. In particular, counter machines can be coded into the grid, so the reachability problem is not decidable.

\begin{convention*} To eliminate bulky exposition, we make the following assumptions
in the rest of the paper.
\begin{itemize}

\item The automata under consideration are viewed as deterministic. Hence, when we write \textquotedblleft automata \textquotedblright, we mean \textquotedblleft deterministic finite automata\textquotedblright.


\item All structures are infinite unless explicitly specified.


\item The graphs are undirected. The case of directed graphs can
be treated in a similar manner.
\end{itemize}
\end{convention*}

\begin{figure}[h]
\begin{center}
\input{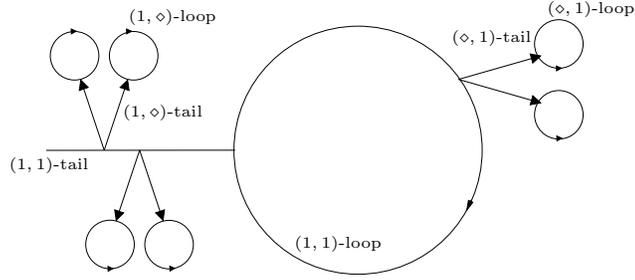}
\caption{\label{fg:standard}A Typical Unary Graph Automaton}
\end{center}
\end{figure}

Let $\G=(V,E)$ be an automatic graph. Let $\A$ be an automaton
recognizing $E$.  We establish some terminology for the automaton
$\A$. The general shape of $\A$ is given in Figure
\ref{fg:standard}. All the states reachable from the initial state
by reading inputs of type $(1,1)$ are called {\bf $(1,1)$-states}.
A {\bf tail} in $\A$ is a sequence of states linked by transitions
without repetition. A {\bf loop} is a sequence of states linked by
transitions such that the last state coincides with the first one,
and with no repetition in the middle. The set of $(1,1)$-states is
a disjoint union of a tail and a loop. We call the tail the {\bf
$(1,1)$-tail} and the loop the {\bf $(1,1)$-loop}. Let $s$ be a
$(1,1)$ state. All the states reachable from $s$ by reading inputs
of type $(1,\diamond)$ are called {\bf $(1,\diamond)$-states}.
This collection of all $(1,\diamond)$-states is also a disjoint
union of a tail and a loop (see the figure), called the {\bf
$(1,\diamond)$-tail} and the {\bf $(1,\diamond)$-loop},
respectively. The {\bf $(\diamond,1)$-tails} and {\bf
$(\diamond,1)$-loops} are defined in a similar matter.

\smallskip

We say that an automaton is {\bf standard} if the lengths of all
its loops and tails equal some number $p$, called the \textbf{loop
constant}. If $\A$ is a standard automaton recognizing a
binary relation, it has exactly $2p$ $(1,1)$-states. On each of
these states, there is a $(1, \diamond)$-tail and a $(\diamond,
1)$-tail of length exactly $p$. At the end of each
$(1,\diamond)$-tail and $(\diamond,1)$-tail there is a
$(1,\diamond)$-loop and $(\diamond, 1)$-loop, respectively, of
size exactly $p$. Therefore if $n$ is the number of states in
$\A$, then $n = 8p^2$.

\begin{lemma}\label{lemma:standard}
Let $\A$ be an $n$ state automaton recognizing a binary relation
$E$ on $1^\star$. There exists an equivalent standard
automaton with at most $8 n^{2n}$ states.
\end{lemma}

\begin{proof} Let $p$ be the least common multiple of the lengths
of all loops and tails of $\A$.  An easy estimate shows that $p$
is no more than $n^n$.  One can transform $\A$ into an equivalent
standard automaton whose loop constant is $p$. Hence, there is a
standard  automaton equivalent to $\A$ whose size is bounded
above by $8n^{2n}$.
\end{proof}

\smallskip

We can simplify the general shape of the automaton using the fact
that we consider undirected graphs. Indeed, we need only consider
transitions labelled by $(\diamond , 1)$.  To see this, given an
automaton with only $(\diamond, 1)$ transitions, to include all
symmetric transitions, add a copy of each $(\diamond, 1)$
transition which is labelled with $(1, \diamond)$.

\smallskip

We recall a characterization theorem of unary automatic graphs
from \cite{rubin}.   Let $\B=(B, E_B)$ and $\D=(D, E_D)$ be finite
graphs.  Let $R_1, R_2$ be subsets of $D \times B$, and $R_3, R_4$
be subsets of $B\times B$.
Consider the graph $\D$ followed by $\omega$ many copies of $\B$,
ordered as $\B^{0}, \B^{1}, \B^{2}, \ldots$.  Formally, the vertex
set of $\B^{i}$ is $B\times\{i\}$ and we write $b^{i}=(b,i)$ for
$b\in B$ and $i \in \omega$. The edge set $E^{i}$ of $\B^{i}$
consists of all pairs $(a^{i}, b^{i})$ such that $(a,b) \in E_B$.
We define the infinite graph, {\bf $unwind(\B,\D, \bar{R})$}, as
follows: \ $1)$\  The vertex set is $D\cup B^{0} \cup B^{1} \cup
B^{2} \cup \ldots$; \ $2)$ \ The edge set contains $E_D \cup
E^{0}\cup E^{1}\cup \ldots$ as well as the following edges, for
all $a,b\in B$, $d\in D$, and $i,j \in \omega$:
\begin{itemize}
\item $(d, b^{0})$ when $(d,b) \in R_1$, and $(d, b^{i+1})$ when
$(d,b) \in R_2$, \item $(a^{i}, b^{i+1})$ when $(a,b) \in R_3$,
and $(a^{i}, b^{i+2+j})$ when $(a,b) \in R_4$.
\end{itemize}

\begin{theorem} \label{thm:characterization} \cite{rubin}
A graph is unary automatic if and only if it is
isomorphic to  $unwind(\B,\D, \bar{R})$ for some parameters $\B$,
$\D$, and $\bar{R}$. Moreover, if $\A$ is a standard automaton
representing $\G$ then the parameters $\B, \D, \bar{R}$ can be
extracted in $O(n^2)$; otherwise, the parameters can be extracted
in $O(n^{2n})$, where $n$ is the number of states in $\A$.
\end{theorem}

\section{Unary automatic graphs of finite degree}

A graph is of  {\bf finite degree} if there are at most finitely
many edges from each vertex $v$. We call an automaton $\A$
recognizing a binary relation over $\{1\}$ a {\bf one-loop
automaton} if its transition diagram contains exactly one loop,
the $(1,1)$-loop. The general structure of one-loop automata is
given in Figure \ref{fg:Type1Auto}.

\begin{figure}[h]
\begin{center}
\resizebox{!}{3cm}{\includegraphics{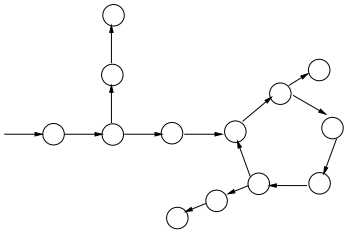}}
\caption{\label{fg:Type1Auto}One-loop automaton}
\end{center}
\end{figure}

\smallskip

We will always assume that the lengths of all the tails of the
one-loop automata are not bigger than the size of the
$(1,1)$-loop. The following is an easy proposition and we omit its
proof.

\begin{proposition}
Let $\G=(V,E)$ be a unary automatic graph, then $\G$ is of finite
degree if and only if there is a one-loop automaton $\A$
recognizing $E$.  \qed
\end{proposition}

By Lemma \ref{lemma:standard}, transforming a given
automaton to an equivalent standard automaton  may blow up the
number of states exponentially. However, there is only polynomial
blow up if $\A$ is a one-loop automaton.

\begin{lemma}
If $\A$ is a one-loop automaton with $n$ states, there exists an
equivalent standard one-loop automaton with loop constant $p\leq n$.
\end{lemma}

\begin{proof} Let $l$ be the length of the loop in $\A$ and $t$ be
the length of the longest tail in $\A$. Let $p$ be the least
multiple of $l$ such that $p\geq t$. It is easy to see that $p
\leq l+t \leq n$. One can transform $\A$ into an equivalent
standard one-loop automaton whose loop constant is $p$.
\end{proof}

Note that the equivalent standard automaton has $2p$ (1,1)-states. From each of them there
is a $(1,\diamond)$-tail of length $p$ and a $(\diamond, 1)$-tail
of length $p$. Hence the automaton has $4p^2$ states. By the above lemma, we always assume the input automaton $\A$ is standard. In the rest of the paper, we will state all results in terms of the loop constant $p$ instead of $n$, the number of states of the input automaton. Since $p\leq n$, for any constant $c>0$, an $O(p^c)$ algorithm can also be viewed as an $O(n^c)$ algorithm.

\smallskip

Given two unary automatic graphs of finite degree $\G_1=(V, E_1)$
and $\G_2=(V, E_2)$ (where we recall the convention that the
domain of each graph is $1^\star$), we can form the {\bf union
graph} $\G_1\oplus \G_2=(V, E_1 \cup E_2)$ and the {\bf
intersection graph} $\G_1\otimes \G_2=(V, E_1 \cap E_2)$.
Automatic graphs of finite degree are closed under these
operations. Indeed, let $\A_1$ and $\A_2$ be one-loop automata
recognizing $E_1$ and $E_2$ with loop constants $p_1$ and $p_2$,
respectively. The standard construction that builds automata for
the union and intersection operations produces a one-loop
automaton whose loop constant is
 $p_1\cdot p_2$. We introduce another operation: consider
the new graph $\G_1'=(V, E_1')$, where the set $E_1'$ of edges is
defined as follows; a  pair $(1^n,1^m)$ is in $E'$ if and only if
$(1^n,1^m) \notin E$ and $|n-m|\leq p_1$. The relation $E_1'$ is
recognized by the same automaton as $E_1$, modified so that  all
$(\diamond,1)$-states that are final declared non-final, and all
the $(\diamond,1)$-states that are non-final declared final. Thus,
we have the following proposition:

\begin{proposition}
If $\G_1$ and $\G_2$ are  automatic graphs of finite degree then
so are $\G_1\oplus \G_2$, $\G_1\otimes \G_2$, and $\G_1'$. \qed
\end{proposition}

Now our goal is to recast Theorem \ref{thm:characterization} for
graphs of finite degree. Our analysis will show that, in contrast
to the general case for automatic graphs,  the parameters $\B$,
$\D$, and $\bar{R}$ for graphs of finite degree can be extracted
in linear time.

\smallskip
\begin{definition}[Unfolding Operation] \label{dfn: unfolding}
Let $\D=(V_{\D}, E_{\D})$ and $\F=(V_{\F}, E_{\F})$ be finite
graphs. Consider the finite sets $\Sigma_{\D,\F}$ consisting of
all mappings $\eta:V_{\D}\rightarrow P(V_{\F})$, and $\Sigma_{\F}$
consisting of all mappings $\sigma: V_{\F}\rightarrow P(V_{\F})$.
Any infinite sequence $\alpha=\eta\sigma_0\sigma_1\ldots$ where
$\eta \in \Sigma_{\D,\F}$ and  $\sigma_i \in \Sigma_{\F}$ for each
$i$, defines the infinite graph $\G_{\alpha}=(V_{\alpha},
E_{\alpha})$ as follows:

\begin{itemize}

\item $V_{\alpha}=V_{\D}\cup \{(v,i) \mid v\in V_{\F},  i \in
\omega\}$.

\item $E_{\alpha}= E_{\D} \cup \{(d, (v,0)) \mid v\in \eta(d)\}
\cup \{((v,i),(v',i)) \mid (v,v') \in E_{\F}, i\in \omega\} \cup
\{((v,i),(v',i+1)) \mid v'\in \sigma_i (v), i \in \omega\}$.

\end{itemize}
\end{definition}

Thus $\G_{\alpha}$ is obtained by taking $\D$ together with an
infinite disjoint union of $\F$ such that edges between $\D$ and
the first copy of $\F$ are put according to the mapping $\eta$,
and edges between successive copies of $\F$ are put according to
$\sigma_i$.

\smallskip

Figure \ref{fg:Gsigma} illustrates the general shape of a unary
automatic graph of finite degree that is build from $\D$, $\F$,
$\eta$, and $\sigma^\omega$, where $\sigma^{\omega}$ is the
infinite word $\sigma\sigma\sigma\cdots$.

\begin{figure}[h]
\begin{center}
\resizebox{!}{2.5cm}{\includegraphics{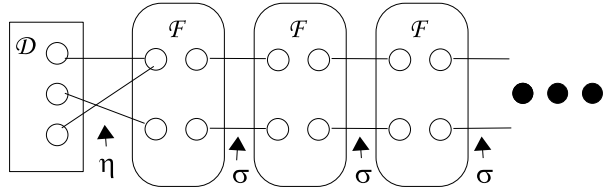}}
\caption{\label{fg:Gsigma}  Unary automatic graph of finite degree
$\G_{\eta\sigma^{\omega}}$}
\end{center}
\end{figure}

\begin{theorem}\label{thm:Gsigma}
A graph of finite degree $\G=(V,E)$ possesses  a unary automatic
presentation if and only if there exist finite graphs $\D, \F$ and
mappings $\eta:V_{\D} \rightarrow P(V_{\F})$ and $\sigma: V_{\F}
\rightarrow P(V_{\F})$ such that $\G$ is isomorphic to
$\G_{\eta\sigma^\omega}$.
\end{theorem}

\begin{proof}
Let $\G=(V,E)$ be a unary automatic graph of finite degree. Let
$\A$ be an automaton recognizing $E$. In linear time on the number
of states of $\A$ we can easily transform $\A$ into a one-loop
automaton. So, we assume that $\A$ is a one-loop automaton with
loop constant $p$.  We construct the finite graph $\D$ by setting
$V_{\D} = \{q_0,q_1,\ldots, q_{p-1}\}$, where $q_0$ is the
starting state, $q_0,\ldots, q_{p-1}$ are all states on the
$(1,1)$-tail such that $q_{i}$ is reached from $q_{i-1}$ by
reading $(1,1)$ for $i>0$; and for $0 \leq i \leq j < p$, $(q_i,
q_j) \in E_{\D}$ iff there is a final state $q_{f}$ on the
$(\diamond,1)$-tail out of $q_{i}$, and the distance from $q_{i}$
to $q_{f}$ is $j-i$. We construct the graph $\F$ similarly by
setting $V_{\F} = \{q'_0,\ldots,q'_{p-1}\}$ where
$q'_0,\ldots,q'_{p-1}$ are all states on the $(1,1)$-loop. The
edge relation $E_{\F}$ is defined in a similar way as $E_{\D}$.
The mapping $\eta: V_{\D} \to P(V_{\F})$ is defined for any $m, n
\in \{0,\ldots,p-1\}$ by putting $q'_n$ in $\eta(q_m)$ if and only
if there exists a final state $q_f$ on the $(\diamond, 1)$-tail
out of $q_m$, and the distance from $q_m$ to $q_f$ equals $p+n-m$.
The mapping $\sigma$ is constructed in a similar manner by reading
the $(\diamond, 1)$-tails out of the $(1,1)$-loop. It is clear
from this construction that the graphs $\G$ and
$G_{\eta\sigma^\omega}$ are isomorphic.

\smallskip

Conversely, consider the graph $\G_{\eta\sigma^\omega}$ for some
$\eta \in \Sigma_{\D}$ and $\sigma \in \Sigma_{\F}$. Assume that
$V_{\D}=\{q_0,\ldots, q_{\ell-1}\}$, $V_{\F}=\{q'_0,\ldots,
q'_{p-1}\}$. A one-loop automaton $\A$ recognizing the edge
relation of $\G_{\eta \sigma^\omega}$ is constructed as follows.
The $(1,1)$-tail of the automaton is formed by
$\{q_{0},\ldots,q_{\ell-1}\}$ and the $(1,1)$-loop is formed by
$\{q'_0,\ldots, q'_{p-1}\}$, both in natural order. The initial
state is $q_0$.  If for some $i<j$, $\{q_{i}, q_{j}\} \in E_{\D}$
, then put a final state $q_{f}$ on the $(\diamond, 1)$-tail
starting from $q_{i}$ such that the distance from $q_{i}$ to
$q_{f}$ is $j-i$. If $q'_{j} \in \eta(q_{i})$, then repeat the
process but make the corresponding distance $p+j-i$. The set of
edges $E_{\F}$ and mapping $\sigma$ are treated in a similar
manner by putting final states on the $(\diamond, 1)$-tails from
the $(1,1)$-loop.

Again, we see that $\A$ represents a unary automatic graph that is
isomorphic to $\G_{\eta\sigma^\omega}$.
\end{proof}
The proof of the above theorem also gives us the following corollary.
\begin{corollary}
If $\G$ is a unary automatic graph of finite degree, the
parameters $\D$, $\F$, $\sigma$ and $\eta$ can be extracted in
$O(p^2)$ time, where $p$ is the loop constant of the one-loop
automaton representing the graph. Furthermore, $|V_\F| = |V_\D| =
p$.\qed\end{corollary}

\section{Deciding the infinite component problem}

Recall the graphs are undirected. A \textbf{component} of $\G$ is
the transitive closure of a vertex under the edge relation. The
{\bf infinite component problem} asks whether a given graph $\G$
has an infinite component.

\begin{theorem}\label{thm:infinite component}
The infinite component problem for unary automatic graph of finite
degree $\G$ is solved in $O(p^{3})$, where $p$ is the loop
constant of the automaton recognizing $\G$.
\end{theorem}

By Theorem \ref{thm:Gsigma}, let $\G=\G_{\eta\sigma^{\omega}}$. We
observe that it is sufficient to consider the case in which
$\D=\emptyset$ (hence $\G=\G_{\sigma^{\omega}}$) since
$\G_{\eta\sigma^\omega}$ has an infinite component if and only if
$\G_{\sigma^\omega}$ has one.

\smallskip

Let $\F^{i}$ be the $i^{th}$ copy of $\F$ in $\G$. Let $x^{i}$ be
the copy of vertex $x$ in $\F^i$. We construct a finite
 directed graph $\F^\sigma = (V^{\sigma},
E^{\sigma})$ as follows. Each node in $V^{\sigma}$ represents a
distinct connected component in $\F$. For simplicity, we assume
that $|V^\sigma| = |V_{\F}|$ and hence use $x$ to denote its own
component in $\F$. The case in which $|V^\sigma| < |V_\F|$ can be
treated in a similar way. For $x, y\in V_{\F}$, put $(x, y) \in
E^\sigma$ if and only if $y' \in \sigma(x')$ for some $x'$ and
$y'$ that are in the same component as $x$ and $y$, respectively.
Constructing  $\F^\sigma$ requires finding connected components of
$\F$ hence takes time $O(p^2)$. To prove the above theorem, we
make essential use of the following definition. See also
\cite{graph book}.


\begin{definition} \label{Dfn:edge-path} An {\bf \em oriented walk} in a directed graph $G$
is a subgraph $\P$ of $G$ that consists of a sequence of nodes
$v_0,...,v_k$ such that for $1\leq i \leq k$, either $(v_{i-1},
v_i)$ or $(v_i, v_{i-1})$ is an arc in $G$, and for each $0\leq
i\leq k$, exactly one of $(v_{i-1},v_i)$ and $(v_i,v_{i-1})$
belongs to $\P$. An oriented walk is an {\bf \em oriented cycle}
if $v_0=v_k$ and there are no repeated nodes in $v_1,...,v_k$.
\end{definition}

In an oriented walk $\P$, an arc $(v_i, v_{i+1})$ is called a {\bf forward arc} and $(v_{i+1},v_i)$ is called a {\bf backward arc}. The {\bf net length} of $\P$, denoted $disp(\P)$, is the difference between the number of forward arcs and backward arcs.
Note the net length can be negative. The next lemma establishes a connection between oriented cycles in $\F^\sigma$ and infinite components in $\G$.




\begin{lemma}\label{lm:cycle}
There is an infinite component in $\G$ if and only if there is an
oriented cycle in $\F^{\sigma}$  such that the net length of the
cycle is positive.
\end{lemma}

\begin{proof}
Suppose there is an oriented cycle $\P$ from $x$ to $x$ in $\F^{\sigma}$ of net length $m>0$.  For all $i\geq p$, $\P$ defines the path $P_i$ in $\G$ from $x^i$ to $x^{i+m}$ where $P_i$ lies in $\F^{i-p}\cup\cdots \cup\F^{i+p}$. Therefore, for a fixed
$i\geq p$, all vertices in the set $\{x^{j m+i} \mid j\in \omega\}$
belong to the same component of $\G$.  In particular, this implies
that $\G$ contains an infinite component.

\smallskip

Conversely, suppose there is an infinite component $D$ in $\G$.
Since $\F$ is finite, there must be some $x$ in $V_\F$ such that
there are infinitely many copies of $x$ in $D$. Let $x^i$ and
$x^j$ be two copies of $x$ in $D$ such that $i < j$. Consider a
path between $x^i$ and $x^j$. We can assume that on this path
there is at most one copy of any vertex $y \in V_\F$ apart from
$x$ (otherwise, choose $x^j$ to be the copy of $x$ in the path
that has this property). By definition of $\G_{\sigma^\omega}$ and
$\F^\sigma$, the node $x$ must be on an oriented cycle of
$\F^\sigma$ with net length $j-i$.
\end{proof}

\begin{proof}[Proof of Theorem \ref{thm:infinite component}]
By the equivalence in Lemma \ref{lm:cycle}, it suffices to provide
an algorithm that decides if $\F^{\sigma}$ contains an oriented
cycle with positive net length.  Notice that the existence of an
oriented cycle with positive net length is equivalent to the
existence of an oriented cycle with negative net length.
Therefore, we give  an algorithm which finds oriented cycles with
non-zero net length.

\smallskip

For each node $x$ in $\F^\sigma$, we search for an oriented cycle
of positive net length from $x$ by creating a labeled queue of
nodes $Q_x$ which are connected to $x$.

\smallskip

\noindent\texttt{ALG:Oriented-Cycle}
\begin{enumerate}

\item Pick node $x \in \F^\sigma$ for which a queue has not been
built yet. Initially the queue $Q_x$ is empty. Let $d(x) = 0$, and
put $x$ into the queue. Mark $x$ as \emph{unprocessed}.  If queues
have been built for each $x \in \F^\sigma$, stop the process and
return \emph{NO}.

\item Let $y$ be the first \emph{unprocessed} node in $Q_x$. If
there are no \emph {unprocessed} nodes in $Q_x$, return to (1).

\item For each of the nodes $z$ in the set $\{z \mid (y,z)\in
E^\sigma \text{ or } (z,y) \in E^\sigma\}$, do the following.

\begin{enumerate}
    \item If $(y,z) \in E^\sigma$, set $d'(z) = d(y) +1$;
     if $(z,y) \in E^\sigma$, set $d'(z) = d(y) -1$. (If both hold, do
     steps (a), (b), (c) first for $(z,y)$ and then for $(y,z)$.)

    \item If $z \notin Q_x$, then set $d(z) = d'(z)$, put $z$
    into $Q_x$, and mark $z$ as \emph{unprocessed}.

    \item If $z \in Q_x$ then
    \begin{enumerate}
       \item if $d(z) = d'(z)$, move to next $z$,
       \item if $d(z) \neq d'(z)$, stop the process and return \emph{YES}.
    \end{enumerate}

\end{enumerate}

\item Mark $y$ as \emph{processed} and go back to (2).

\end{enumerate}

An important property of this algorithm is that when we are
building a queue for node $x$ and are processing $z$, both $d(z)$
and $d'(z)$ represent net lengths of paths from $x$ to $z$.

We claim that the algorithm returns \emph{YES} if and only if
there is an oriented cycle in $\F^\sigma$ with non-zero net
length. Suppose the algorithm returns \emph{YES}.  Then, there is
a base node $x$ and a node $z$ such that $d(z) \neq d'(z)$. This
means that there is an oriented walk $\P$ from $x$ to $z$ with net
length $d(z)$ and there is an oriented walk $\P'$ from $x$ to $z$
with net length $d'(z)$. Consider the oriented walk $\P\overleftarrow{\P'}$, where
$\overleftarrow{\P'}$ is the oriented walk $\P'$ in reverse
direction. Clearly this is an oriented walk from $x$ to $x$ with
net length $d(z) -d'(z) \neq 0$. If there are no repeated nodes in
$\P\overleftarrow{\P'}$, then it is the required oriented cycle.
Otherwise, let $y$ be a repeated node in $\P\overleftarrow{\P'}$
such that no nodes between the two occurrences of $y$ are
repeated. Consider the oriented walk between these two occurrences
of $y$, if it has a non-zero net length, then it is our required
oriented cycle; otherwise, we disregard the part between the two
occurrences of $z$ and make the oriented walk shorter without
altering its net length.

\smallskip

Conversely, suppose there is an oriented cycle $\P =
x_0,\ldots,x_m$ of non-zero net length where $x_0=x_m$. However,
we assume for a contradiction that the algorithm returns
\emph{NO}. Consider how the algorithm acts when we pick $x_0$ at
step (1). For each $0\leq i\leq m$, one can prove the following
statements by induction on $i$.

\begin{enumerate}
\item[$(\star$)] $x_i$ always gets a label $d(x_i)$

\item[$(\star\star$)] \ \ $d(x_i)$ equals the net length of the
oriented walk from $x_0$ to $x_i$ in $\P$.
\end{enumerate}

By the description of the algorithm, $x_0$ gets the label $d(x_0)
= 0$. Suppose the statements holds for $x_i$, $0\leq i<m$, then at
the next stage, the algorithm labels all nodes in $\{z \mid (z,
x_i) \in E^\sigma \text{ or } (x_i,z)\in E^\sigma\}$. In
particular, it calculates $d'(x_{i+1})$. By the inductive
hypothesis, $d'(x_{i+1})$ is the net length of the oriented walk
from $x_0$ to $x_{i+1}$ in $\P$. If $x_{i+1}$ has already had a
label $d(x_{i+1})$ and $d(x_{i+1}) \neq d'(x_{i+1})$, then the
algorithm would return \emph{YES}. Therefore $d(x_{i+1}) =
d'(x_{i+1})$.  By assumption on $\P$, $d(x_{m}) \neq 0$. However,
since $x_0 = x_m$, the induction gives that $d(x_m) = d(x_0) = 0$.
This is a contradiction, and thus the above algorithm is correct.

\smallskip

In summary, the following algorithm solves the infinite component
problem. Suppose we are given an automaton (with loop
constant $p$) which recognizes the  unary automatic graph of
finite degree $\G$. Recall that $p$ is also the cardinality of
$V_\F$. We first compute $\F^\sigma$, in time $O(p^2)$. Then we
run \texttt{Oriented-Cycle} to decide whether $\F^\sigma$ contains an
oriented cycle with positive net length. For each node $x$ in
$\F^\sigma$, the process runs in time $O(p^2)$. Since $\F^\sigma$
contains $p$ number of nodes, this takes time $O(p^3)$.

\smallskip

Note that Lemma \ref{lm:cycle} holds for the case when $|V_\F| >
|\F / \sim_{comp}|$. Therefore the algorithm above can be slightly
modified to apply to this case as well.
\end{proof}

\section{Deciding the infinity testing problem}

We next turn our attention to the {\bf infinity testing problem}
for unary automatic graphs of finite degree.  Recall that this
problem asks for an algorithm that, given a vertex $v$ and a graph
$\G$, decides if $v$ belongs to an infinite component. We prove the
following theorem.

\begin{theorem}\label{thm:InfTest}
The infinity testing problem for unary automatic graph of finite
degree $\G$ is solved in $O(p^3)$, where $p$ is the loop constant
of the automaton $\A$ recognizing $\G$. In particular, when
$\A$ is fixed, there is a constant time algorithm that decides the
infinity testing problem on $\G$.
\end{theorem}

For a fixed input $x^i$, we have the following lemma.

\begin{lemma}\label{lm:infinite}
If $x^i$ is connected to some $y^j$ such that $|j-i| > p$, then $x^i$ is in an infinite component.
\end{lemma}

\begin{proof} Suppose such a $y^j$ exists.
Take a path $P$ in $\G$ from $x^i$ to $y^j$. Since $p$ is the cardinality of $V_\F$, there is $z\in V_\F$ such that $z^s$ and $z^t$ appear in $P$ with $s<t$. Therefore all nodes in the set $\{z^{s+(t-s)m}\mid m\in \omega\}$ are in the same component as $x^i$.
\end{proof}

Let $i'=min\{p,i\}$. To decide if $x^i$ and $y^j$ are in the same component, we run a breadth first search in $\G$ starting from $x^i$ and going through all vertices in $\F^{i-i'},\ldots, \F^{i+p}$. The algorithm is as follows:

\smallskip
\noindent\texttt{ALG: FiniteReach}
\begin{enumerate}

\item Let $i'= \min\{p, i\}$.

\item Initialize the queue $Q$ to be empty. Put the pair $(x,0)$ into $Q$ and mark it as \emph{unprocessed}.

\item If there are no \emph{unprocessed} pairs in $Q$, stop the process. Otherwise, let $(y,d)$ be the first \emph{unprocessed}
pair.  For arcs $e$ of the form $(y,z)$ or $(z,y)$ in $E^\sigma$, do the following.

\begin{enumerate}

\item If $e$ is of the form $(y,z)$, let $d'=d+1$; if $e$ is of the form $(z,y)$, let $d'=d-1$.

\item If $-i'\leq d'\leq p$  and $(z,d')$ is not in $Q$, then put $(z,d')$ into $Q$ and mark $(z,d')$ as \emph{unprocessed}.

\end{enumerate}

\item Mark $(y,d)$ as \emph{processed}, and go to (2).

\end{enumerate}

Note that any $y^j$ is reachable from $x^i$ on the graph $\G$ restricted on $\F^{i-i'}, \ldots, \F^{i+p}$ if and only if after running \texttt{FiniteReach} on the input $x^i$, the pair $(y,j-i)$ is in $Q$.
When running the algorithm we only use the exact value of the input $i$ when $i< p$ (we set $i'=p-1$ whenever $i\geq p$), so the running time of \texttt{FiniteReach} is bounded by the number of edges in $\G$ restricted to $\F^0,\ldots, \F^{2p}$. Therefore the running time is $O(p^3)$. Let $B = \{y \mid (y,p) \in Q\}$.

\begin{lemma} \label{lm:infTest-B}
Let $x\in V_\F$. $x^i$ is in an infinite component if and only if $B\neq \emptyset$.
\end{lemma}

\begin{proof}
Suppose a vertex $y\in B$, then there is a path from $x^i$ to $y^{i+p}$. By Lemma \ref{lm:infinite}, $x^i$ is in an infinite component.
Conversely, if $x^i$ is in an infinite component, then there must be some vertices in $\F^{i+p}$ reachable from $x^i$. Take a path from $x^i$ to a vertex $y^{i+p}$ such that $y^{i+p}$ is the first vertex in $\F^{i+p}$ appearing on this path. Then $y\in B$.
\end{proof}

\begin{proof}[Proof of Theorem \ref{thm:InfTest}]
We assume the input vertex $x^i$ is given by tuple $(x,i)$. The above lemma suggests a simple algorithm to check if $x^i$ is in an infinite component.

\smallskip
\noindent\texttt{ALG: InfiniteTest}
\begin{enumerate}

\item Run \texttt{FiniteReach} on vertex $x^{i}$, computing the set $B$ while building the queue $Q$.

\item For every $y\in B$, check if there is edge $(y,z) \in E^\sigma$. Return $YES$ if one such edge is found; otherwise, return $No$.

\end{enumerate}

Running \texttt{FiniteReach} takes $O(p^3)$ and checking for edge $(y,z)$ takes $O(p^2)$.  The running time is therefore $O(p^3)$. Since $x$ is bounded by $p$, if $\A$ is fixed, checking whether $x^i$ belongs to an infinite component takes constant time.
\end{proof}


\section{Deciding the reachability problem}

Suppose $\G$ is a unary automatic graph of finite degree
represented by an automaton with loop constant $p$. The {\bf
reachability problem} on $\G$ is formulated as: given two vertices
$x^i,y^j$ in $\G$, decide if $x^i$ and $y^j$ are in the same
component. We prove the following theorem.


\begin{theorem}\label{thm:reachability}
Suppose $\G$ is a unary automatic graph of finite degree
represented by an automaton $\A$ of loop constant $p$. There
exists a polynomial-time algorithm that solves the reachability
problem on $\G$. For inputs $u,v$, the running time of the
algorithm is $O(|u| + |v| + p^{4})$.
\end{theorem}

We restrict to the case when $\G=\G_{\sigma^\omega}$.  The proof can
be modified slightly to work in the more general case, $\G =
\G_{\eta \sigma^\omega}$.

\smallskip

Since, by Theorem \ref{thm:InfTest}, there is an $O(p^3)$-time
algorithm to check if $x^{i}$ is in a finite component, we can
work on the two possible cases separately. We first deal with the
case when the input $x^i$ is in a finite component.
By Lemma \ref{lm:infinite}, $x^i$ and $y^j$ are in the same
(finite) component if and only if after running
\texttt{FiniteReach} on the input $x^i$, the pair $(y,j-i)$ is in
the queue $Q$.

\begin{corollary} \label{cr:finite reach}
If all components of $\G$ are finite and we represent $(x^i, y^j)$
as $(x^i, y^j, j-i)$, then there is an $O(p^3)$-algorithm deciding
if $x^i$ and $y^j$ are in the same component.\qed
\end{corollary}

Now, suppose that $x^i$ is in an infinite component. We start with
the following question: given $y\in V_\F$, are $x^i$ and $y^i$ in
the same component in $\G$?  To answer this, we present an
algorithm that computes all vertices $y\in V_\F$ whose $i^{th}$ copy
lies in the same $\G$-component as $x^i$. The algorithm is similar
to \texttt{FiniteReach}, except that it does not depend on the
input $i$. Line(3b) in the algorithm is changed to the following:

\smallskip

\indent(3b) If $  -p\leq d'\leq p$ and $(z,d')$ is not in $Q$,
then put $(z,d')$ into $Q$ and mark $(z,d')$ as
\emph{unprocessed}.

\smallskip

We use this modified algorithm to define the set $Reach(x) = \{y
\mid (y,0) \in Q\}$. Intuitively, we can think of the algorithm as
a breadth first search through $\F^{0} \cup \cdots \cup \F^{2p}$
which originates at $x^p$. Therefore, $y\in Reach(x)$ if and only
if there exists a path from $x^p$ to $y^p$ in $\G$ restricted to
$\F^{0} \cup \cdots \cup \F^{2p}$.

\begin{lemma} \label{lm:reach_reach}
Suppose $x^i$ is in an infinite component. The vertex $y^i$ is in
the same component as $x^i$ if and only if $y^i$ is also in an
infinite component and $y \in Reach(x)$.
\end{lemma}

\begin{proof}
Suppose $y^i$ is in an infinite component and $y \in Reach(x)$. If
$i\geq p$, then the observation above implies that there is a path
from $x^i$ to $y^i$ in $\F^{i-p} \cup \cdots \cup \F^{i+p}$. So,
it remains to prove that $x^i$ and $y^i$ are in the same component
even if $i<p$.

\smallskip

Since $y\in Reach(x)$, there is a path $P$ in $\G$ from $x^p$ to $y^p$. Let $\ell$ be the least number such that $\F^\ell \cap P
\neq \emptyset$. If $i \geq p-\ell$, then it is clear that $x^i$ and $y^i$ are in the same component.  Thus, suppose that
$i<p-\ell$. Let $z$ be such that $z^\ell \in P$. Then $P$ is $P_1P_2$ where $P_1$ is a path from $x^p$ to $z^\ell$ and $P_2$
is a path from $z^\ell$ to $y^p$.   Since $x^i$ is in an infinite component, it is easy to see that $x^p$ is also in an infinite component.  There exists an $r>0$ such that all vertices in the set $\{x^{p+rm} \mid m\in \omega\}$ are in the same component. Likewise, there is an $r'>0$
such that all vertices in $\{y^{p+r'm} \mid m\in \omega\}$ are in the same component. Consider $x^{p+rr'}$ and $y^{p+rr'}$. Analogous to
the path $P_1$, there is a path $P_1'$ from $x^{p+rr'}$ to $z^{\ell+rr'}$.  Similarly, there is a path $P_2'$ from
$z^{\ell+rr'}$ to $y^{p+rr'}$. We describe another path $P'$ from $x^p$ to $y^p$ as follows. $P'$ first goes from $x^p$ to
$x^{p+rr'}$, then goes along $P_1'P_2'$ from $x^{p+rr'}$ to $y^{p+rr'}$ and finally goes to $y^p$. Notice that the least
$\ell'$ such that $\F_{\ell'} \cap P' \neq \emptyset$ must be larger than $\ell$.  We can iterate this procedure of lengthening
the path between $x^p$ and $y^p$ until $i< p-\ell'$, as is required to reduce to the previous case.

\smallskip

To prove the implication in the other direction, we assume that
$x^i$ and $y^i$ are in the same infinite component. Then $y^i$ is,
of course, in an infinite component. We want to prove that $y\in
Reach(x)$. Let $i'=\min\{p, i\}$. Suppose there exists a path $P$
in $\G$ from $x^i$ to $y^i$ which stays in $\F^{i-i'}\cup \cdots
\cup \F^{i+p}$. Then, indeed, $y\in Reach(x)$. On the other hand,
suppose no such path exists.  Since $x^i$ and $y^i$ are in the
same component, there is some path $P$ from $x^i$ to $y^i$. Let
$\ell(P)$ be the largest number such that $P\cap \F^{\ell(P)} \neq
\emptyset$. Let $\ell'(P)$ be the least number such that $P\cap
\F^{\ell'(P)} \neq \emptyset$. We are in one of two cases:
$\ell(P) > i+p$ or $\ell'(P) < i-p$.  We will prove that if
$\ell(P) > i+p$ then there is a path $P'$ from $x^i$ to $y^i$ such
that $\ell(P') < \ell(P)$ and $\ell'(P') \geq i-p$. The case in
which $\ell'(P) < i-p$ can be handled in a similar manner.

\smallskip

Without loss of generality, we assume $\ell'(P) =i$ since
otherwise we can change the input $x$ and make $\ell'(P) = i$. Let
$z$ be a vertex in $\F$ such that $z^{\ell(P)} \in P$. Then $P$ is
$P_1P_2$ where $P_1$ is a path from $x^i$ to $z^{\ell(p)}$ and
$P_2$ is a path from $z^{\ell(p)}$ to $y^i$. Since $\ell(P) >i+p$,
there must be some $s^j$ and $s^{j+k}$ in $P_1$ such that $k>0$.
For the same reason, there must be some $t^m$ and $t^{m+n}$ in
$P_2$ such that $n>0$. Therefore, $P$ contains paths between any
consecutive pair of vertices in the sequence $(x^i, s^j, s^{k+j},
z^p, t^{m+n}, t^m, y^i)$. Consider the following sequence of
vertices:
\[
(x^i, s^j, t^{m+n-k},t^{m-k}, s^{j-n}, s^{j+k-n}, t^m, y^i).
\]
It is easy to check that there exists a path between each pair of
consecutive vertices in the sequence. Therefore the above sequence
describes a path $P'$ from $x^i$ to $y^i$. It is easy to see that
$\ell(P') = \ell(P) - n$. Also since $\ell'(P) = i$, $\ell'(P') >
i-p$. Therefore $P'$ is our desired path.
\end{proof}

\smallskip

In the following, we abuse notation by using $Reach$ and $\sigma$
on subsets of $V_{\F}$.  We inductively define a sequence
$Cl_0(x), Cl_1(x), \ldots $ such that each $Cl_k(x)$ is a subset
of $V_{\F}$.  Let $Cl_0(x)=Reach(x)$ and For $k>0$, we define
$Cl_k(x) = Reach(\sigma(Cl_{k-1}(x)))$. The following lemma is
immediate from this definition.

\begin{lemma} \label{lm:reach_closure}
Suppose $x^i$ is in an infinite component, then $x^i$ and $y^j$
are in the same component if and only if $y^j$ is also in an
infinite component and $y\in Cl_{j-i}(x)$.\qed
\end{lemma}

We can use the above lemma to construct a simple-minded algorithm
that solves the reachability problem on inputs $x^i, y^j$.

\smallskip

\noindent\texttt{ALG: Na\"iveReach}
\begin{enumerate}

\item Check if each of $x^i$, $y^j$ are in an infinite component
of $\G$ (using the algorithm of Theorem \ref{thm:InfTest}).

\item If exactly one of $x^i$ and $y^j$ is in a finite component,
then return \emph{NO}.

\item If both $x^i$ and $y^j$ are in finite components, then run
\texttt{FiniteReach} on input $x^i$ and check if $(y,j-i)$ is in
$Q$.

\item If both $x^i$ and $y^j$ are in infinite components, then
compute $Cl_{j-i}(x)$. If $y\in Cl_{j-i}(x)$, return \emph{YES};
otherwise, return \emph{NO}.

\end{enumerate}

We now consider the complexity of this algorithm. The set
$Cl_0(x)$ can be computed in time $O(p^3)$.  Given $Cl_{k-1}(x)$,
we can compute $Cl_k(x)$ in time $O(p^3)$ by computing $Reach(y)$ for any $y\in \sigma(Cl_{k-1}(x))$. Therefore, the total
running time of \texttt{Na\"iveReach} on input $x^i$, $y^j$ is
$(j-i)\cdot p^3$. We want to replace the multiplication with
addition and hence tweak the algorithm.
\smallskip

From Lemma \ref{lm:infTest-B}, $x^i$ is in an infinite component in $\G$ if and only if \texttt{FiniteReach} finds a vertex $y^{i+p}$ connecting to $x^i$. Now, suppose that $x^i$ is in an infinite component. We can use \texttt{FiniteReach} to find such a $y$, and a path from $x^i$ to $y^{i+p}$.
On this path, there must be two vertices $z^{i+j}, z^{i+k}$ with $0 \leq j<k\leq p$. Let $r=k-j$. Note that $r$ can be computed from the algorithm. It is easy to see that all vertices in the set $\{x^{i+mr} \mid m\in \omega\}$ belong to the same component.


\begin{lemma}\label{lm:reach_repeat}
$Cl_0(x) = Cl_r(x)$.
\end{lemma}

\begin{proof}
By definition, $y\in Cl_0(x)$ if and only if $x^p$ and $y^p$ are
in the same component of $\G$. Suppose that there exists a path in
$\G$ from $x^p$ to $y^p$. Then there is a path from $x^{p+r}$ to
$y^{p+r}$. Since $x^p$ and $x^{p+r}$ are in the same component of
$\G$, $x^p$ and $y^{p+r}$ are in the same component. Hence $y\in
Cl_r(x)$.

\smallskip

For the reverse inclusion, suppose $y\in Cl_r(x)$. Then there
exists a path from $x^p$ to $y^{p+r}$. Therefore, $x^{p+r}$ and
$y^{p+r}$ are in the same component. Since $r\leq p$, $x^p$ and
$y^p$ are in the same component.
\end{proof}

\smallskip

Using the above lemma, we define a new algorithm \texttt{Reach} on
inputs $x^i$, $y^j$ by replacing line (4) in \texttt{Na\"iveReach}
with

\smallskip

\indent(4) If $x^i$ and $y^j$ belong to infinite components, then
compute  $Cl_0(x),\ldots, Cl_{r-1}(x)$. If $y\in Cl_k(x)$ for
$k<r$ such that $j-i = k \mod r$, return \emph{YES}; otherwise,
return \emph{NO}.

\smallskip

\begin{proof}[Proof of Theorem \ref{thm:reachability}]
Say input vertices are given as $x^i$ and $y^j$. By Lemma
\ref{lm:reach_closure} and Lemma \ref{lm:reach_repeat}, the
algorithm \texttt{Reach} returns \emph{YES} if and only if $x^i$
and $y^j$ are in the same component. Since $r\leq p$, calculating
$Cl_0(x),\ldots, Cl_{r-1}(x)$ requires time $O(p^4)$. Therefore
the running time of \texttt{Reach} on input $x^i$, $y^j$ is
$O(i+j+p^4)$.
\end{proof}

\smallskip

Notice that, in fact, the algorithm produces a number $k<p$ such
that in order to check if $x^i$, $y^j$ ($j>i$) are in the same
component, we need to test if $j-i<p$ and if $j-i=k \mod p$.
Therefore if $\G$ is fixed and we compute $Cl_0(x),\ldots,
Cl_{r_x-1}(x)$ for all $x$ beforehand, then deciding whether two
vertices $u$, $v$ belong to the same component takes linear time.
The above proof can also be used to build an automaton that
decides reachability uniformly:

\begin{corollary}\label{cr:reach_Aut}
Given a unary automatic graph of finite degree $\G$ represented by
an automaton with loop constant $p$, there is a deterministic
automaton with at most $2p^4+p^3$ states that solves the
reachability problem on $\G$. The time required to construct this
automaton is $O(p^5)$.
\end{corollary}

\begin{proof}
For all $0\leq x<p$, $i\in \omega$, let string $1^{i p+x}$
represent vertex $x^i$ in $\G$. Suppose $ip+x\leq jp+y$, we
construct an  automaton $\A_{Reach}$ that accepts $(1^{ip+x},
1^{jp+y})$ if and only if $x^i$ and $y^j$ are in the same
component in $\G$.

\begin{enumerate}

\item $\A_{Reach}$ has a $(1,1)$-tail of length $p^2$. Let the
states on the tail be $q_0,q_1,\ldots,q_{p^2-1}$, where $q_0$ is
the initial state. These states represent vertices in $\F^0,
\F^1,\ldots, \F^{p-1}$.

\item From $q_{p^2-1}$, there is a $(1,1)$-loop of length $p$. We
call the states on the loop $q_0',q_1',\ldots,q_{p-1}'$. These
states represent vertices in $\F^{p}$.

\item For $0\leq x,i< p$, there is a $(\diamond, 1)$-tail from
$q_{i p+x}$ of length $p^2-x$. We denote the states on this tail
by $q_{i p+x}^1,\ldots,q_{i p+x}^{p^2-x}$. These states represent
vertices in $\F^{i},\F^{i+1},\ldots, \F^{i+p-1}$.

\item For $0\leq x,i\leq p$, if $x^i$ is in an infinite component,
then there is a $(\diamond, 1)$-loop of length $r \times p$ from
$q_{i p+x}^{p^2-x}$. The states on this loop are called
$\check{q}_{i p+x}^{1},\ldots,\check{q}_{i p+x}^{r  p}$. These
states represent vertices in $\F^{i+p},\ldots, \F^{i+p+r-1}$.

\item For $0\leq x\leq p$, if $x^p$ is in a finite component, then
there is a $(\diamond, 1)$-tail from $q_{x}'$ of length $p^2$.
These states are denoted
$\hat{q}_{x}^{1},\ldots,\hat{q}_{x}^{p^{2}}$ and represent vertices
in $\F_{p},\ldots, \F_{2p-1}$.

\item If $x^p$ is in an infinite component, from $q_{x}'$, there
is a $(\diamond, 1)$-loop of length $r \times p$. We write these
states as $\tilde{q}_{x}^1,\ldots,\tilde{q}_{x}^{r  p}$.

\end{enumerate}

The final (accepting) states of $\A_{Reach}$ are defined as
follows:
\begin{enumerate}

\item States $q_0,\ldots, q_{p^2-1}, q_0',\ldots,q_{p-1}$ are
final.

\item For $i<p$, if $x^i$ is in a finite component, run the
algorithm \texttt{FiniteReach} on input $x^i$ and declare state
$q_{i p+x}^{j p+y-x}$ final if $(y, j)\in Q$.

\item For $i<p$, if $x^i$ is in an infinite component, compute
$Cl_0(x),\ldots,Cl_{r-1}(x)$.

    \begin{enumerate}

    \item Make state $q_{i p+x}^{j p+y-x}$ final if
    $y^{i+j}$ is in an infinite component and $y\in Cl_j(x)$.

    \item Make state $\check{q}_{i p+x}^{j p+y-x}$ final if $y\in Cl_j(x)$

    \end{enumerate}

\item If $x^p$ is in a finite component, run the algorithm
\texttt{FiniteReach} on input $x^p$ and make state $\hat{q}_{x}^{j
p+y-x}$ final if $(y, j)\in Q$.

\item If $x^p$ is in an infinite component, compute
$Cl_0(x),\ldots,Cl_{r-1}(x)$. Declare state $\tilde{q}_x^{j
p+y-x}$ final if $y\in Cl_j(x)$.

\end{enumerate}

One can show that $\A_{Reach}$ is the desired automaton. To
compute the complexity of building $\A_{Reach}$, we summarize the
computation involved.

\begin{enumerate}

\item For all $x^i$ in $\F^0\cup\cdots\cup\F^p$, decide whether
$x^i$ is in a finite component. This takes time $O(p^5)$ by Theorem
\ref{thm:InfTest}.

\item For all $x^i$ in $\F^0\cup\cdots\cup\F^p$ such that $x^i$ is
in a finite component, run \texttt{FiniteReach} on input $x^i$.
This takes time $O(p^5)$ by Corollary \ref{cr:finite reach}.

\item For all $x\in V_\F$ such that $x^p$ is in an infinite
component, compute the sets $Cl_0(x),$ $\ldots, Cl_{r-1}(x)$. This
requires time $O(p^5)$ by Theorem \ref{thm:reachability}.

\end{enumerate}

Therefore the running time required to construct $\A_{Reach}$ is
$O(p^5)$.
\end{proof}

\section{Deciding the connectivity problem}
Finally, we present a solution to the {\bf connectivity problem} on unary automatic graphs of finite degree. Recall a graph is \textbf{connected} if there is a path between any pair of vertices. The construction of $\A_{Reach}$ from the last section suggests an immediate solution to the connectivity problem.

\smallskip

\noindent\texttt{ALG: Na\"iveConnect}
\begin{enumerate}
\item Construct the automaton $\A_{Reach}$.
\item Check if all states in $\A_{Reach}$ are final states. If it is the case, return $YES$; otherwise, return $NO$.
\end{enumerate}

The above algorithm takes time $O(p^5)$. Note that $\A_{Reach}$ provides a uniform solution to the reachability problem on $\G$. Given the ``regularity'' of the class of infinite graphs we are studying, it is reasonable to believe there is a more intuitive algorithm that solves the connectivity problem. It turns out that this is the case.

\begin{theorem}\label{thm:connectivity}
The connectivity problem for unary automatic graph of finite
degree $\G$ is solved in $O(p^3)$, where $p$ is the loop constant
of the automaton recognizing $\G$.
\end{theorem}

Observe that if $\G$ does not contain an infinite component, then $\G$ is not connected. Therefore we suppose $\G$ contains an infinite component $C$.

\begin{lemma}\label{lm:infinite-connect}
For all $i\in \N$, there is a vertex in $\F^i$ belonging to $C$.
\end{lemma}

\begin{proof}
Since $C$ is infinite, there is a vertex $x^i$ and $s>0$ such that all vertices in $\{x^{i+ms}\mid m\in \omega\}$ belong to $C$ and $i$ is the least such number. By minimality, $i<s$. Take a walk along the path from $x^{i+s}$ to $x^{i}$. Let $y^s$ be the first vertex in $\F^s$ that appears on this path.  It is easy to see that $y^0$ must also be in $C$. Therefore $C$ has a non-empty intersection with each copy of $\F$ in $\G$.
\end{proof}

Pick an arbitrary $x \in V_{\F}$ and run \texttt{FiniteReach} on $x^0$ to compute the queue $Q$.
Set $R = \{y \in V_\F \mid (y,0) \in Q\}$.

\begin{lemma}
Suppose $\G$ contains an infinite component, then $\G$ is connected if and only if $R = V_\F$.
\end{lemma}

\begin{proof}
Suppose there is a vertex $y \in V_\F - R$. Then there is no path in $\G$ between $x^0$ to $y^0$. Otherwise, we can shorten the path from $x^0$ to $y^0$ using an argument similar to the proof of Lemma \ref{lm:reach_reach}, and show the existence of a path between $x^0$ to $y^0$ in the subgraph restricted on $\F^0, \ldots, \F^p$. Therefore $\G$ is not connected. \ Conversely, if $R=V_\F$, then every set of the form $\{ y\in V_\F \mid (y,i) \in Q\}$ for $i\geq 0$ equals $V_\F$. By Lemma \ref{lm:infinite-connect}, all vertices are in the same component.
\end{proof}

\begin{proof}[Proof of Theorem \ref{thm:connectivity}]
By the above lemma the following algorithm decides the connectivity problem on $G$:

\smallskip

\noindent\texttt{ALG: Connectivity}
\begin{enumerate}
\item Use the algorithm proposed by Theorem \ref{thm:infinite component} to decide if there is an infinite component in $G$. If there is no infinite component, then stop and return $NO$.

\item Pick an arbitrary $x \in V_{\F}$, run \texttt{FiniteReach} on $x^0$ to compute the queue $Q$.

\item Let $C = \{y \mid (y, 0)\in Q\}$. If $C = V_{\F}$, return $YES$; otherwise, return $NO$.
\end{enumerate}

Solving the infinite component problem takes $O(p^3)$ by Theorem \ref{thm:infinite component}. Running algorithm \texttt{FiniteReach} also takes $O(p^3)$. Therefore \texttt{Connectivity} takes $O(p^3)$. \end{proof}

\section{Conclusion}

In this paper we addressed algorithmic problems for graphs of
finite degree that have automata presentations over a unary alphabet. We provided
polynomial-time algorithms that solve connectivity, reachability,
infinity testing, and infinite component problems. In our future
work we plan to improve these algorithms for other stronger
classes of unary automatic graphs. We also point out that there
are many other algorithmic problems for finite  graphs that can be
studied for the class of unary automatic graphs. These, for
example, may concern finding spanning trees for automatic graphs,
studying the isomorphism problems, and other related issues.

\end{document}